\documentclass[12pt,draft]{article}
\usepackage[dvipdfmx]{graphicx}
\usepackage{amsfonts}
\usepackage{amsmath,amsthm}
\usepackage{enumerate}

\usepackage{amssymb}
\usepackage{mathrsfs}
\allowdisplaybreaks[4]

\newtheorem{thm}{Theorem}[section]
\newtheorem{lem}[thm]{Lemma}
\newtheorem{pro}[thm]{Proposition}

\numberwithin{equation}{section}

\begin{document}
\setlength{\abovedisplayskip}{6pt}
\setlength{\belowdisplayskip}{6pt}

\title{\textbf{Approximate functional equation for the derivatives of functions in Selberg class}%\\[0.5em]{\large\textbf{(1st draft) }\\[-1em]}
}
\author{Yoshikatsu Yashiro}
\date{}

\maketitle
\vspace{-2em}
\begin{abstract}
Let $F(s)$ be a function belonging to Selberg class. Chandrasekharan and Narasiman proved the approximate functional equation for $F(s)$. In this paper, we shall generalize this formula for the derivatives of $F(s)$.
\end{abstract}

\section{Introduction}
Selberg \cite{SEL} introduced a class of zeta functions satisfying the following properties: 
\begin{enumerate}[(a)]
\item The function $F(s)$ is written as an absolutely convergent Dirichlet series $F(s)=\sum_{n=1}^\infty a_F(n)n^{-s}$ for ${\rm Re\;}s>1$.  
\item There exists $m\in\mathbb{Z}_{\geq0}$ such that $(s-1)^mF(s)$ is an entire function of finite order.  
\item The function $\Phi(s)=Q^s\prod_{j=1}^q\Gamma(\lambda_js+\mu_j)F(s)$ satisfies $\Phi(s)=\omega\overline{\Phi}(1-s)$ where $q\in\mathbb{Z}_{\geq1}$, $Q\in\mathbb{R}_{>0}$, $\lambda_j\in\mathbb{R}_{>0}$, $\mu_j\in\mathbb{C}: {\rm Re\;}\mu_j\geq0$, $\omega\in\mathbb{C}: |\omega|=1$, and $\overline{X}$ denotes $\overline{X}(s)=\overline{X(\overline{s})}$.  
\item The Dirichlet coefficients $a_F(n)$ satisfy $a_F(n)=O(n^\varepsilon)$ for any $\varepsilon\in\mathbb{R}_{>0}$.
\item The function $\log F(s)$ is written as $\log{F(s)}=\sum_{n=1}^\infty{b_F(n)n^{-s}}$ where $b_F(n)$ satisfy $b_F(n)=0$ for $n\ne p^r\;(r\in\mathbb{Z}_{\geq1})$ and $b_F(n)=O(n^\theta)$ for some $\theta\in\mathbb{R}_{<1/2}$, 
\end{enumerate}
which is called the Selberg class $\mathcal{S}$. For example, we see that the Riemann zeta function $\zeta(s)=\sum_{n=1}^\infty n^{-s}$ and the $L$-function attatched to cusp forms $L_f(s)=\sum_{n=1}^\infty\lambda_f(n)n^{-s}$ belong to $\mathcal{S}$, where $f$ is a cusp form of weight $k$ given by $f(z)=\sum_{n=1}^\infty\lambda_f(n)n^{\frac{k-1}{2}}e^{2\pi inz}$. By Landau's \cite{LAN} result (see (15) of p.214), the conditions (a)--(d) of $\mathcal{S}$ imply that the average of $a_F(s)$ is approximated as 
\begin{align}
\sum_{n\leq x}a_F(n)
=\begin{cases} x\sum_{r=0}^{p_F-1}c_r(\log x)^r+O(x^{\frac{d_F-1}{d_F+1}+\varepsilon}), & p_F\in\mathbb{Z}_{\geq1}, \\ O(x^{\frac{d_F-1}{d_F+1}+\varepsilon}), & p_F\in\mathbb{Z}_{\leq0}, \end{cases} \label{LANA}
\end{align}
where $c_r\in\mathbb{C}$, $p_F$ is the order of pole at $s=1$ for $F(s)$, and $d_F$ is given by $d_F=2\sum_{j=1}^q\lambda_j$ called the degree of $F$. The conditon {\rm(c)} implies that the $m$-th derivatives of $F(s)$ holds 
\begin{align}
F^{(m)}(s)=\sum_{r=0}^m(-1)^r\binom{n}{r}\chi_F^{(m-r)}(s)\overline{F}^{(r)}(1-s) \label{FDFE}
\end{align}
for $s\in\mathbb{C}$ where $\chi_F(s)$ is given by
\begin{align}
\chi_F(s)=\omega Q^{1-2s}\prod_{j=1}^q\frac{\Gamma(\lambda_j(1-s)+\overline{\mu_j})}{\Gamma(\lambda_js+\mu_j)}. \label{FCHI}
\end{align}

Chandrasekharan and Narasiman \cite{C&N} proved the approximate functional equation for a class of zeta functions (see Theorem 2 of p.53). In the Selberg class, this equation is written as
\begin{align}
F(s)=\sum_{n\leq y}\frac{a_F(n)}{n^s}+\chi_F(s)\sum_{n\leq y}\frac{a_F(n)}{n^{1-s}}+O(|t|^{\frac{d_F}{2}(1-\sigma)-\frac{1}{d_F}}(\log y)^{p_F}) \label{AFES}
\end{align}
under the condition $a_n\geq0$ where $F\in\mathcal{S}$, $s=\sigma+it: 0\leq\sigma\leq 1, |t|\geq1$ and $y=(Q\lambda_1^{\lambda_1}\cdots\lambda_q^{\lambda_q})|t|^{{d_F}/{2}}$. 
On the other hand, the author \cite{YY3} showed the approximate functional equation for the derivatives of $L$-function attached for cusp form as
\begin{align*}
L_f^{(m)}(s)=\sum_{n\leq\frac{|t|}{2\pi}}\frac{\lambda_f(n)(-\log n)^m}{n^s}+\chi_{L_f}(s)\sum_{n\leq\frac{|t|}{2\pi}}\frac{\lambda_f(n)(-\log n)^m}{n^{1-s}}+O(|t|^{1/2-\sigma+\varepsilon}),
\end{align*}
and introduced the mean value formula for $L_f^{(m)}(s)$ as
\begin{align*}
&\int_0^T|L_f^{(m)}(\sigma+it)|^2dt\\&=
\begin{cases}
A_fT(\log T)^{2m+1}+O(T(\log T)^{2m}), & \sigma=1/2,\\ \displaystyle T\sum_{n=1}^\infty\frac{|\lambda_f(n)|^2(\log n)^{2m}}{n^{2\sigma}}+O(T^{2(1-\sigma)}(\log T)^{2m}), & \sigma\in(1/2,1), \\ \displaystyle T\sum_{n=1}^\infty\frac{|\lambda_f(n)|^2(\log n)^{2m}}{n^{2\sigma}}+O((\log T)^{2m+2}), & \sigma=1. 
\end{cases}
\end{align*}
where $\chi_{L_f}(s)=(-1)^{k/2}(2\pi)^{1-2s}\Gamma(1-s+\tfrac{k-1}{2})/\Gamma(s+\tfrac{k-1}{2})$ and $A_f$ is a positive constant depending on $f$. These results are generalized for those results for $L_f(s)$ proved by Good \cite{GD1}. 
For $F\in\mathcal{S}, \sigma>1/2$ and $T>0$, let $N_{F^{(m)}}(\sigma,T)$ be a number of zeros for $F^{(m)}(s)$ in the region ${\rm Re\;}s\geq\sigma$, $0<{\rm Im\;}s\leq T$. 
As the application of those results, the author \cite{YY4} obtained a zero-density estimate for $L_f^{(m)}(s)$, which is 
\begin{align*}
N_{L_f^{(m)}}(\sigma,T)
=O\left(\frac{T}{\sigma-1/2}\log\frac{1}{\sigma-1/2}\right) \quad (T\to\infty)
\end{align*}
for $\sigma>1/2$. This result corresponds the estimate of zero-density estimate for $\zeta^{(m)}(s)$
\begin{align*}
N_{\zeta^{(m)}}(\sigma,T)=O\left(\frac{T}{\sigma-1/2}\log\frac{1}{\sigma-1/2}\right) \quad (T\to\infty)
\end{align*}
for $\sigma>1/2$ which is shown by Aoki--Minamide \cite{A&M}. However, zero-density estimates for derivatives for zeta-function belonging to Selberg class is not known.

In this paper we study the approximate functional equation for $F^{(m)}(s)$ in order to establish tools for estimating $N_{F^{(m)}}(\sigma,T)$ in next paper. To attain the above object we shall use Good's \cite{GD1} method and the result \eqref{LANA}.  
Let $\mathcal{R}$ be a class of $C^\infty$-class functions $\varphi:[0,\infty)\to\mathbb{R}$ satisfying $\varphi(\rho)=1$ for $\rho\in[0,1/2]$ and $\varphi(\rho)=0$ for $\rho\in[2,\infty)$, which is called characteristic functions. Put $\varphi_0(\rho):=1-\varphi(1/\rho)$ and $\|\varphi^{(j)}\|_1=\int_0^\infty|\varphi^{(j)}(\rho)|d\rho$ where $\varphi^{(j)}$ is the $j$-th derivative function of $\varphi\in\mathcal{R}$. 
Then we see that $\varphi_0\in\mathcal{R}$ and $\|\varphi^{(j)}\|_1<\infty$. For $j\in\mathbb{Z}_{\geq0}$, $r\in\{0,\dots,m\}$, $\rho\in\mathbb{R}_{>0}$ and $s\in\mathbb{C}: |t|\gg1$, we define 
\begin{align}
{{\gamma}_{j}^{(r)}}(s;\rho):=&\frac{1}{2\pi i}\int_{\mathcal{F}}\frac{{g_{F}}(s+w)}{{g_{F}}(s)}\frac{1}{w\cdots(w+j)}{\frac{\chi_F^{(r)}}{\chi_F}}(1-(s+w))\times\notag\\
&\times\prod_{j=1}^q\frac{\Gamma(\lambda_j(s+w)+{\mu_j})}{\Gamma(\lambda_js+{\mu_j})}(\rho e^{-\frac{\pi}{2}i{\rm sgn\;}(t)})^{\frac{d_F}{2}w}dw, \label{DEFGM}
\\{{\delta}_{j}^{(r)}}(s;\rho):=&\frac{1}{2\pi i}\int_{\mathcal{F}}\frac{\overline{g_{F}}(s+w)}{\overline{g_{F}}(s)}\frac{1}{w\cdots(w+j)}{\frac{\chi_F^{(r)}}{\chi_F}}(1-(s+w))\times\notag\\&\times\prod_{j=1}^q\frac{\Gamma(\lambda_j(s+w)+\overline{\mu_j})}{\Gamma(\lambda_js+\overline{\mu_j})}(\rho e^{-\frac{\pi}{2}i{\rm sgn\;}(t)})^{\frac{d_F}{2}w}dw \label{DEFDM}
\end{align}
where ${\rm sgn}(t)=\pm1$ if $t\gtrless 0$, $\mathcal{F}=\{(1/2\pm 1)-\sigma+\sqrt{|t|}e^{i\pi(\mp1/2+\theta)}\mid \theta\in[0,1]\}\cup\{u-\sigma\pm i\sqrt{|t|}\mid u\in[-1/2,3/2]\}$ (double-sign corresponds), and $g_F(s)$ is given by
\begin{align}
g_{F}(s)=&\left(s(1-s)\right)^{p_F+m}\textstyle\prod_{j=1}^q(A_j(s)\overline{A_j}(1-s))^{m+1}. 
\label{DEFG}
\end{align}
Here $A_{j}(s)$ is given by
\begin{align}
A_j(s)&=\begin{cases}  1, & {\rm Re\;}\mu_j>\lambda_j/2, \\[-0.2em] \prod_{n=0}^{\nu_j}(\lambda_js+\mu_j-n),  & {\rm Re\;}\mu_j\leq\lambda_j/2 \end{cases} \label{DEFA} 
\end{align}
where $\nu_j=0$ when ${\rm Re\;}\mu_j>\lambda_j/2$ and $\nu_j=[\lambda_j/2-{\rm Re\;}\mu_j]$ when ${\rm Re\;}\mu_j\leq\lambda_j/2$. 

Then we obtain the approximate functional equation for $F^{(m)}(s)$ containing characteristic functions:
\begin{thm}\label{THM1}
For any $F\in\mathcal{S}$, $\varphi\in\mathcal{R}$, $m\in\mathbb{Z}_{\geq0}$, $l\in\mathbb{Z}_{>M_F}$, $s\in\mathbb{C}: \sigma\in[0,1], |t|\gg1$ and $y_1,y_2\in\mathbb{R}_{>0}: y_1y_2=(Q\lambda_1^{\lambda_1}\cdots\lambda_q^{\lambda_q})^2|t|^{d_F}$, we have  
\begin{align*}
F^{(m)}(s)=&\sum_{n=1}^\infty\frac{a_F(n)(-\log n)^m}{n^s}\varphi\left(\frac{n}{y_1}\right)+\\
&+\sum_{r=0}^m(-1)^r\binom{m}{r}\chi_f^{(m-r)}(s)\sum_{n=1}^\infty\frac{\overline{a_F(n)}(-\log n)^m}{n^{1-s}}\varphi_0\left(\frac{n}{y_2}\right)+R_\varphi(s)
\end{align*}
where $M_F$ is some positive constant and $R_\varphi(s)$ is given by
\begin{align*}
&\hspace{-1.2em}R_\varphi(s)\\
:=&\sum_{\frac{y_1}{2}\leq n\leq 2y_1}\frac{a_F(n)(-\log n)^m}{n^s}\sum_{j=1}^l\varphi^{(j)}\left(\frac{n}{y_1}\right)\left(-\frac{n}{y_1}\right)^r\gamma_{j}^{(m)}\left(s;\frac{1}{(\lambda_1^{\lambda_1}\cdots\lambda_q^{\lambda_q})^{\frac{2}{d_F}}|t|}\right)\times\\
&+\chi_F(s)\sum_{r=0}^m(-1)^r\binom{m}{r}\sum_{\frac{y_2}{2}\leq n\leq 2y_2}\frac{\overline{a_F(n)}(-\log n)^r}{n^{1-s}}\times\\
&\times\sum_{j=1}^l\varphi^{(j)}\left(\frac{n}{y_2}\right)\left(-\frac{n}{y_2}\right)^j{\delta_{j}^{(r)}}\left(1-s;\frac{1}{(\lambda_1^{\lambda_1}\cdots\lambda_q^{\lambda_q})^{\frac{2}{d_F}}|t|}\right)+\\
&+O\left(y_1^{1-\sigma}(\log y_1)^{m+\max\{p_F-1,0\}}|t|^{-\frac{l}{2}}\|\varphi^{(l+1)}\|_1\right)+\\
&+O\left(y_2^\sigma|t|^{d_F(\frac{1}{2}-\sigma)-\frac{l}{2}}\|\varphi_0^{(l+1)}\|_1\sum_{r=0}^m(\log y_2)^{r+\max\{p_F-1,0\}}(\log|t|)^{m-r}\right).
\end{align*}
\end{thm}

Next introducing new functions $\varphi_\alpha\in\mathcal{R}$ and $\xi\not\in\mathcal{R}$, replacing $\varphi$ to $\varphi_\alpha$ in the above theorem, and choosing $\alpha\in\mathbb{R}_{>0}$ in order to minimize the error term, we obtain the approximate functional equation for $F(s)$:
\begin{thm}\label{THM2}
For any $F\in\mathcal{S}$, $m\in\mathbb{Z}_{\geq0}$ and $s=\sigma+it: 0\leq\sigma\leq1, |t|\gg1$ we have 
\begin{align}
F^{(m)}(s)=&\sum_{n\leq y_1}\frac{a_F(n)(-\log n)^m}{n^s}+\sum_{r=0}^m(-1)^r\binom{m}{r}\chi_F^{(m-r)}(s)\sum_{n\leq y_2}\frac{\overline{a_F(n)}(-\log n)^r}{n^{1-s}}+\notag\\
&+O(y_1^{1-\sigma+\varepsilon}|t|^{-\frac{1}{2}})+O(y_2^{\sigma+\varepsilon}|t|^{d_F(\frac{1}{2}-\sigma)-\frac{1}{2}}) \label{NCAFE}
\end{align}
where $y_1, y_2\in\mathbb{R}_{>0} : y_1y_2=(Q\lambda_1^{\lambda_1}\cdots\lambda_q^{\lambda_q})^2|t|^{d_F}$. 
By choosing \\$y_1=y_2=(Q\lambda_1^{\lambda_1}\cdots\lambda_q^{\lambda_q})|t|^{d_F/2}$, the error terms of \eqref{NCAFE} are $O(|t|^{\frac{d_F}{2}(1-\sigma)-\frac{1}{2}+\varepsilon})$. 
\end{thm}

As an example of Theorem \ref{THM2}, we shall give the approximate functional equation for derivatives of Rankin-Selberg $L$-function. Let $f$ and $g$ be holomorphic cusp forms of weight $k$ with respect to $SL_2(\mathbb{Z})$ given by  $f(z)=\sum_{n=1}^\infty\lambda_f(n)n^{\frac{k-1}{2}}e^{2\pi inz}$ and $g(z)=\sum_{n=1}^\infty\lambda_g(n)n^{\frac{k-1}{2}}e^{2\pi inz}$, the Rankin-Selberg $L$-function is defined by $$L_{f\times g}(s)=\sum_{n=1}^\infty\frac{\lambda_f(n)\overline{\lambda_g(n)}}{n^{-s}}.$$ By results of Rankin \cite{RAN}, Selberg \cite{SEL2} and Delinge \cite{DEL}, we see that $L_{f\times g}(s)$ belongs to $\mathcal{S}$. Hence we obtain 
\begin{align*}
L_{f\times g}^{(m)}(s)=&\sum_{n\leq\frac{|t|^2}{4\pi^2}}\frac{\lambda_{f\times g}(n)(-\log n)^m}{n^s}+\\
&+\sum_{r=0}^m(-1)^r\binom{m}{r}\chi_{L_{f\times g}}^{(m-r)}(s)\sum_{n\leq\frac{|t|^2}{4\pi^2}}\frac{{\overline{\lambda_{f\times g}(n)}}(-\log n)^r}{n^{1-s}}+O(|t|^{\frac{3}{2}-2\sigma+\varepsilon})
\end{align*}
where $m\in\mathbb{Z}_{\geq0}$, $s=\sigma+it: \sigma\in[0,1], |t|\gg1$, and $\chi_{L_{f\times g}}(s)$ is given by 
$$\chi_{L_{f\times g}}(s)=(2\pi)^{4s-2}\frac{\Gamma({1-s})\Gamma(1-s+k-1)}{\Gamma(s)\Gamma(s+k-1)}.$$ 

 In next section we shall show preliminary lemmas to prove Theorems \ref{THM1} and \ref{THM2}. Using these lemmas we shall give Theorems \ref{THM1} and \ref{THM2} in Section 3 and 4 respectively. 

\newpage

\section{Some Lemmas}\label{LEMP}

First for $\varphi\in\mathcal{R}$ we define 
\begin{align*}
K_\varphi(w):=&w\int_0^\infty\varphi(\rho)\rho^{w-1}d\rho\quad({\rm Re\;}w>0). 
\end{align*}
Then $K_\varphi(w)$ has the following properties: 
\begin{lem}[{\cite[LEMMA 3]{GD1}}]\label{KPL}
The function $K_\varphi(w)$ is analytically continued to the whole $w$-plane. Furthermore $K_\varphi(w)$ has the functional equation 
\begin{align}
K_\varphi(w)=K_{\varphi_0}(-w). \label{KPW1}
\end{align}
and the integral representation 
\begin{align}
K_\varphi(w)=\frac{(-1)^{l+1}}{(w+1)\cdots(w+l)}\int_0^\infty\varphi^{(l+1)}(\rho)\rho^{w+l}d\rho \label{KPW2}
\end{align}
for any $l\in\mathbb{Z}_{\geq0}$. Especially $K_\varphi(0)=1$.  
\end{lem}

Next to approximate $(\chi_F^{(r)}/\chi_F)(s)$, we shall use the following lemma: 
\begin{lem}[{\cite[Lemma 2.3]{YY3}}]\label{LIB}
Let $F$ and $G$ be holomorphic functions in the region in $D$ satisfying $\log F(s)=G(s)$ and $F(s)\ne0$ for $s\in D$. Then for any $r\in\mathbb{Z}_{\geq1}$ there exist $\ell_1,\dots,\ell_r\in\mathbb{Z}_{\geq0}$ and $C_{\ell_1,\dots,\ell_r}\in\mathbb{Z}_{\geq0}$ such that 
\begin{align}
\frac{F^{(r)}}{F}(s)
=\sum_{1\ell_1+\cdots+r\ell_r=r}C_{\ell_1,\dots,\ell_r}(G^{(1)}(s))^{\ell_1}\cdots(G^{(r)}(s))^{\ell_r}.
\end{align}
for $s\in D$. Especially $C_{r,0\cdots,0}=1$. 
\end{lem}

Before approximating $(\chi_F^{(r)}/\chi_F)(s)$ we shall show the following formulas:
\begin{lem}\label{HAFF}
Let $D=\{z\in\mathbb{C} \mid {\rm Re\:}z<\delta, |{\rm Im\:}z|<1\}$ where $\delta\in\mathbb{R}_{\geq0}$. 
For any $s\in\mathbb{C}\setminus D$ we satisfy the following formulas: 
\begin{enumerate}[(i)]
\item $\displaystyle \sum_{n=1}^\infty\frac{1}{(s+n)^l}=\begin{cases} O(|t|^{-(l-1)}), & |t|\gg1, \\ O(1),  & |t|\ll1, \end{cases}$ \ where $l\in\mathbb{Z}_{\geq2}$.
\item $\displaystyle \sum_{n=1}^\infty\left(\frac{1}{s+n}-\frac{1}{n}\right)=\begin{cases} -\log|t|+i{\pi}{\rm sgn}(t)/2-\gamma+O(|t|^{-1}),& |t|\gg1, \\ O(1), & |t|\ll1,  \end{cases}$\\ 
where $\gamma$ is the Euler's constant given by $\gamma=1-\int_1^\infty\{u\}u^{-2}du$. 
\end{enumerate}
\end{lem}
\begin{proof}
Using partial summation we have   
\begin{align}
\sum_{n=1}^\infty\frac{1}{(s+n)^l}=
l\int_1^\infty\frac{u-\{u\}}{(s+u)^{l+1}}du
\ll \int_1^\infty\left(\frac{1}{|s+u|^l}+\frac{|s|}{|s+u|^{l+1}}\right)du. \label{HP11}
\end{align}
The estimates
\begin{align}
|u+s|\geq \begin{cases}|{\rm Im}(u+s)|\geq |s\sin\varepsilon|, & u\leq|s| {\rm\:and\:}\arg s\in[-\pi+\varepsilon,\pi-\varepsilon], \\ |{\rm Re}(u+s)|\geq |s|\cos\varepsilon, &  u\leq|s| {\rm\:and\:}\arg s\in[-\varepsilon,\varepsilon],\\ |{\rm Re}(u+s)|\geq u, &  u\geq|s|, \end{cases} \label{IFTE}
\end{align}
give that the right-hand side of \eqref{HP11} is
\begin{align}
&\ll\int_1^{|s|}\left(\frac{1}{|s\sin\varepsilon|^l}+\frac{|s|}{|s\sin\varepsilon|^{l+1}}\right)du+\int_{|s|}^\infty\left(\frac{1}{u^l}+\frac{|s|}{u^{l+1}}\right)du
\ll\frac{1}{|s|^{l-1}} \label{HP12}
\end{align}
for any $s\in\mathbb{C}\setminus D$ and $l\in\mathbb{Z}_{\geq2}$. 
Combining \eqref{HP11} and \eqref{HP12}, and estimating $|s|^{-(l-1)}$ we obtain the formula (i). Similally to (i), we calculate  
\begin{align}
&\sum_{n=1}^\infty \left(\frac{1}{s+n}-\frac{1}{n}\right)\notag\\
&=\int_1^\infty\left(\frac{1}{(s+u)^2}-\frac{1}{u^2}\right)(u-\{u\})du\notag\\
&=\int_1^\infty\left(\frac{1}{s+u}-\frac{1}{u}\right)du+\int_1^\infty\frac{(-s)}{(s+u)^2}du-\int_1^\infty\frac{\{u\}}{(s+u)^2}du+1-\gamma. \label{HP21}
\end{align}
Then for $s\in\mathbb{C}\setminus D$ the first and second term of right-hand side of \eqref{HP21} are 
\begin{align}
&=-\log(s+1)\notag\\
&=-\log s+O(|s|^{-1})=\begin{cases} -\log|t|+i\pi{\rm sgn}(t)/2+O(|t|^{-1}), & |t|\gg1, \\ O(1), & |t|\ll1 \end{cases} \label{HP22}
\end{align}
and
\begin{align}
=-1+\frac{1}{1+s}=\begin{cases} -1+O(|t|^{-1}), & |t|\gg1, \\ O(1), & |t|\ll1. \end{cases} \label{HP23}
\end{align}
respectively, where \eqref{IFTE} and $\log s=\log|t|+i\pi{\rm sgn}(t)/2+O(|t|^{-1})$ (see \cite[p.335]{GD1}) were used. 
By combining \eqref{HP21}--\eqref{HP23} the formula (ii) is obtained.
\end{proof}

Using the infinite product of $\Gamma(s)$ and applying the above lemma with $F=\chi_F$, we obtain the approximate formula for $(\chi_F^{(r)}/\chi_F)(s)$ as follows: 
\begin{lem}\label{CHIDF}
For any $F\in\mathcal{S}$ and $r\in\mathbb{Z}_{\geq0}$ the function $(\chi_F^{(r)}/\chi_F)(s)$ has pole of order $r$ at $s=-(\mu_j+n)/\lambda_j,\:1+(\overline{\mu_j}+n)/\lambda_j$ where $n\in\mathbb{Z}_{\geq0}$ and $j\in\left\{1,\dots,q\right\}$. Put
\begin{align*}
D&=\{z\in\mathbb{C}\mid a\leq{\rm Re\;}s\leq b\}, \\
E_1&=\left\{z\in\mathbb{C}\:\left|\:{\rm Re\;}z<\max_{j}\tfrac{-{\rm Re\;}\mu_j+\delta}{\lambda_j},\:\:\min_{j}\tfrac{-{\rm Im\;}\mu_j-1}{\lambda_j}<{\rm Im\;}z<\max_{j}\tfrac{-{\rm Im\;}\mu_j+1}{\lambda_j}\right.\right\},\\
E_2&=\left\{z\in\mathbb{C}\:\left|\:{\rm Re\;}z>\max_{j}\tfrac{{\rm Re\;}\mu_j-\delta}{\lambda_j}-1,\:\:\min_{j}\tfrac{{\rm Im\;}\mu_j-1}{\lambda_j}+1<{\rm Im\;}z<\max_{j}\tfrac{{\rm Im\;}\mu_j+1}{\lambda_j}+1\right.\right\} 
\end{align*}
where $a,b\in\mathbb{R}$ and $\delta\in\mathbb{R}_{>0}$. Then for any $s\in\mathbb{C}\setminus(E_1\cup E_2)$ we have 
\begin{align*}
\frac{\chi_F^{(r)}}{\chi_F}(s)=\begin{cases} \left(-\log(C_F|t|^{d_F})\right)^m+O\left(\dfrac{(\log|t|)^{m-1}}{|t|}\right), & |t|\gg1, \\ O(1), & |t|\ll1 \end{cases}
\end{align*}
where $C_F=(Q\lambda_1^{\lambda_1}\cdots\lambda_q^{\lambda_q})^{2}$. 
\end{lem}
\begin{proof}
First we check the location and order of pole for $(\chi_F^{(m)}/\chi_F)(s)$. Taking logarithmic differentiation in the both-hand side of $(1/\Gamma)(s)=se^{\gamma s}\prod_{n=1}^\infty(1+s/n)e^{-s/n}$ and \eqref{FCHI} we have  
\begin{align}
-\frac{\Gamma'}{\Gamma}(s)&=\frac{1}{s}+\gamma+\sum_{n=1}^\infty\left(\frac{1}{s+n}-\frac{1}{n}\right),  \label{DFCH2}\\
\frac{\chi_F'}{\chi_F}(s)&=-2\log Q-\sum_{j=1}^q\lambda_j\left(\frac{\Gamma'}{\Gamma}(\lambda_js+\mu_j)+\frac{\Gamma'}{\Gamma}(\lambda_j(1-s)+\overline{\mu_j})\right) \label{DFCH1}
\end{align}
respectively. Put $G^{(l)}(s)=(d^{l-1}/ds^{l-1})G^{(1)}(s)$ and let $G^{(1)}(s)$ be the right-hand side of \eqref{DFCH1}. Applying \eqref{DFCH2} to \eqref{DFCH1}, we get 
\begin{align}
G^{(l)}(s)=\begin{cases}
\displaystyle-2\log Q+d_F\gamma+\sum_{j=1}^q\lambda_j\Biggl(\frac{1}{\lambda_js+\mu_j}+\frac{1}{\lambda_j(1-s)+\overline{\mu_j}}+ & \\
\displaystyle+\sum_{n=1}^\infty\left(\frac{1}{\lambda_js+\mu_j+n}-\frac{1}{n}+\frac{1}{\lambda_j(1-s)+\overline{\mu_j}+n}-\frac{1}{n}\right)\Biggr), & l=1, \\
\displaystyle \sum_{j=1}^q\lambda_j^l\sum_{n=0}^\infty\left(\frac{(-1)^{l-1}(l-1)!}{(\lambda_js+\mu_j+n)^l}+\frac{(l-1)!}{(\lambda_j(1-s)+\overline{\mu_j}+n)^l}\right), & l\in\mathbb{Z}_{\geq2}.
\end{cases}\label{GLSAF}
\end{align}
Hence $G^{(l)}(s)$ has pole of order $l$ at $s=-(\mu_j+n)/\lambda_j,\;1+(\overline{\mu_j}+n)/\lambda_j\;(n\in\mathbb{Z}_{\geq0})$. By using Lemma \ref{LIB} the first statement of Lemma \ref{CHIDF} is showed. 

Lastly we shall approximate $G^{(l)}(s)$. Since ${\rm Re}(\lambda_js+\mu_j)<\delta$, $|{\rm Im}(\lambda_js+\mu_j)|<1$ for $s\in E_1$ and ${\rm Re}(\lambda_j(1-s)+\mu_j)<\delta$, $|{\rm Im}(\lambda_j(1-s)+\overline{\mu_j})|<1$ for $s\in E_2$, Lemma \ref{HAFF} gives that
\begin{align*}
G^{(j)}(s)=\begin{cases} 
\displaystyle O(|t|^{-(j-1)}), & j\in\mathbb{Z}_{\geq2}, |t|\gg1, \\ O(1), & j\in\mathbb{Z}_{\geq1}, |t|\ll1. \end{cases}
\end{align*}
for $s\in\mathbb{C}\setminus(E_1\cup E_2)$. Especially in the case of $j=1$ and $|t|\gg1$, since ${\rm sgn}({\rm Im}(\lambda_js+\mu_j))+{\rm sgn}({\rm Im}(\lambda_j(1-s)+\overline{\mu_j}))=0$, $G^{(1)}(s)$ is approximated as
\begin{align*}
G^{(1)}(s)&=-2\log Q+d_F\gamma+\sum_{j=1}^q\lambda_j(-2\log|\lambda_jt+{\rm Im\;}\mu_j|-2\gamma)+O(|t|^{-1})\\
&=-\log\bigl((Q\lambda_1^{\lambda_1}\cdots \lambda_q^{\lambda_q})^2|t|^{d_F}\bigr)+O(|t|^{-1})
\end{align*}
for $s\in\mathbb{C}\setminus(E_1\cup E_2)$ and $|t|\gg1$. By Lemma \ref{LIB} a desired approximation is obtained:
\begin{align*}
\frac{\chi_F^{(m)}}{\chi_F}(s)=&(G^{(1)}(s))^m+O(|G^{(1)}(s)|^{m-2}|G^{(2)}(s)|)\\
=&\begin{cases} (-d_F\log(C_F|t|))^m+O(|t|^{-1}{(\log t)^{m-1}}), & |t|\gg1, \\  O(1), & |t|\ll1. \end{cases}
\end{align*}
\end{proof}

Next in order to estimate the gamma-factors of \eqref{DEFGM} and \eqref{DEFDM}, we shall use the following estimates:
\begin{lem}[{\cite[Lemma 2 of p.334]{GD1}}]%\label{GFEST}
For $a,b\in\mathbb{R}$ put $D:=\{z\in\mathbb{C}\mid a\leq{\rm Re\:}z\leq b\}$ and $E_{-}:=\{z\in\mathbb{C}\mid {\rm Re\:}z<1/2, |{\rm Im\:}z|<1\}$. Then for any fixed $s\in D$ and $C_0\in\mathbb{R}_{>0}$ we have 
\begin{align*}
\left|\frac{\Gamma(s+w)}{\Gamma(s)}
(e^{-i\frac{\pi}{2}{\rm sgn}(t)})^{w}\right|
\leq\begin{cases}\displaystyle C_1\frac{(1+|t+v|)^{\sigma+u-1/2}}{|t|^{\sigma-1/2}}, & \text{if\;} s+w\in D\setminus E_{-}, \\ C_2|t|^{u}, & \text{if\;} |w|\leq C_0\sqrt{|t|}. \end{cases}
\end{align*}
where $C_1$ and $C_2$ are constants. 
\end{lem}
Replace $s\mapsto\lambda_js+\mu_j$ and $w\mapsto\lambda_jw$ in the above lemma. Using the trivial estimate $(1+|\lambda_j(t+v)+{\rm Im\:}\mu_j|)^{\lambda_j(\sigma+u)+{\rm Re\:}\mu_j-1/2}\asymp \lambda_j^{\lambda_ju}(1+|t+v|)^{\lambda_j(\sigma+u)+{\rm Re\:}\mu_j-1/2}$ for $s\in D$ and $s+w\in D\setminus E_{1}$, and multiplying the above formula for $j\in\{1,\dots,q\}$, a desired estimate is obtained:
\begin{lem}\label{GFEST}
Let $D$ and $E_1$ be those of Lemma \ref{CHIDF} respectively. Then for any fixed $s\in D$ and $c_0\in\mathbb{R}_{>0}$ we have %
\begin{align*}
&\left|\prod_{j=1}^q\frac{\Gamma(\lambda_j(s+w)+\mu_j)}{\Gamma(\lambda_js+\mu_j)}
(e^{-i\frac{\pi}{2}{\rm sgn}(t)})^{\frac{d_F}{2}w}\right|\\
&\leq\begin{cases}\displaystyle c_1(\lambda_1^{\lambda_1}\cdots\lambda_q^{\lambda_q})^{u}\frac{(1+|t+v|)^{\frac{d_F(\sigma+u)+e_F-q}{2}}}{|t|^{\frac{d_F\sigma+e_F-q}{2}}}, & \text{if\;}s+w\in D\setminus E_1,\\ 
c_2(\lambda_1^{\lambda_1}\cdots\lambda_r^{\lambda_r}|t|^{\frac{d_F}{2}})^u, & \text{if\;}  |w|\leq c_0\sqrt{|t|}, \end{cases}
\end{align*}
where $e_F:=2\sum_{j=1}^q{\rm Re\;}\mu_j$ and $c_1$, $c_2$ are constants depending on $\lambda_1, \mu_1, \dots, \lambda_q, \mu_q$. 
\end{lem}

By using Stirling's formula and residue theorem the functions $\gamma_{j}^{(r)}(s;\rho)$ and \\ $\delta_{j}^{(r)}(s;\rho)$ are approximate as follows: 
\begin{lem}\label{RCDE}
For any $j, r\in\mathbb{Z}_{\geq0}$ and $s\in\mathbb{C}: |t|\gg1$ we have 
\begin{align}
{\gamma_{j}^{(r)}}\left(s;\frac{1}{(\lambda_1^{\lambda_1}\cdots\lambda_q^{\lambda_q})^{\frac{2}{d_F}}|t|}\right)=&\begin{cases} \displaystyle (\chi_F^{(r)}/\chi_F)(1-s), & j=0, \\ 
\displaystyle O({|t|^{-1}}(\log|t|)^r), &  j=1, \\ \displaystyle O({|t|^{-\frac{j}{2}}}(\log|t|)^r), & j\in\mathbb{Z}_{\geq2}. \end{cases} \label{GJEE}
\end{align}
The function ${\delta}_{j}^{(r)}(s;(\lambda_1^{\lambda_1}\cdots\lambda_q^{\lambda_q})^{-\frac{2}{d_F}}|t|^{-1})$ equals the right hand side of \eqref{GJEE}.  
\end{lem}
\begin{proof}
Since $|w|\ll\sqrt{|t|}$ for $w\in\mathcal{F}$, from Lemma \ref{GFEST} we can obtain a desired formula in the case of $j\in\mathbb{Z}_{\geq2}$:
\begin{align*}
\gamma_{j}^{(r)}\left(s;\frac{1}{(\lambda_1^{\lambda_1}\cdots\lambda_q^{\lambda_q})^{\frac{2}{d_F}}|t|}\right)\ll\int_{\mathcal{F}}\frac{(\log|t|)^{r}}{|t|^{\frac{j+1}{2}}}|dw|\ll\frac{(\log|t|)^{r}}{|t|^{\frac{j}{2}}}. %\label{DJEE}
\end{align*} 
Using Cauchy's residue theorem we have $\gamma_0^{(r)}(s,\rho)=(\chi_F^{(r)}/\chi_F)(1-s)$ and 
\begin{align}
&\gamma_{1}^{(r)}\left(s;\frac{1}{(\lambda_1^{\lambda_1}\cdots\lambda_q^{\lambda_q})^{\frac{2}{d_F}}|t|}\right)\notag\\
&={\frac{\chi_F^{(r)}}{\chi_F}}(1-s)-\frac{{g_F}(s-1)}{{g_F}(s)}\frac{\chi_F^{(r)}}{\chi_F}(2-s)\prod_{j=1}^q\frac{\Gamma(\lambda_j(s-1)+{\mu_j})}{\Gamma(\lambda_js+{\mu_j})}(\lambda_jit)^{\lambda_j} \label{G1A1}
\end{align}
where we used $|t|e^{\frac{\pi}{2}i{\rm sgn}(t)}=it$. It is clear that  
\begin{align}
\frac{{g_F}(s-1)}{{g_F}(s)}=\frac{1+O(|s|^{-1})}{1+O(|s|^{-1})}=1+O\left(\frac{1}{|t|}\right). \label{G1A2} 
\end{align}
Stirling's formula $\Gamma(s)=\sqrt{2\pi}s^{s-1/2}e^{-s}(1+O(|s|^{-1}))$ and the trivial approximation $\lambda_js+\mu_j=i\lambda_jt(1+O(|t|^{-1}))$ give that
\begin{align}
\frac{\Gamma(\lambda_j(s-1)+\mu_j)}{\Gamma(\lambda_js+{\mu_j})}
&=\frac{(i\lambda_jt)^{\lambda_j(s-1)+{\mu_j}-1/2}e^{-i\lambda_jt}(1+O(|t|^{-1}))}{(i\lambda_jt)^{\lambda_js+{\mu_j}-1/2}e^{-i\lambda_jt}(1+O(|t|^{-1}))}\notag\\
&=(\lambda_jit)^{-\lambda_j}(1+O(|t|^{-1})), \label{G1A3}
\end{align}
Combining \eqref{G1A1}--\eqref{G1A3} and using Lemma \ref{CHIDF} we obtain 
\begin{align*}
\gamma_{1}^{(r)}\left(s;\frac{1}{(\lambda_1^{\lambda_1}\cdots\lambda_q^{\lambda_q})^{\frac{2}{d_F}}|t|}\right)
=&{\frac{\chi_F^{(r)}}{\chi_F}}(1-s)-\frac{\chi_F^{(r)}}{\chi_F}(2-s)+O\left(\frac{1}{|t|}\left|\frac{\chi_F^{(r)}}{\chi_F}(2-s)\right|\right)\notag\\
=&O(|t|^{-1}{(\log|t|)^{r}}).%\label{G1AX}
\end{align*}
By the same discussion, the approximate formula of ${\delta}_{j}^{(r)}(s;1/((\lambda_1^{\lambda_1}\cdots\lambda_q^{\lambda_q})^{\frac{2}{d_F}}|t|))$ is also obtained.    
\end{proof}

Finally, in order to prove Theorem \ref{THM2} from Theorem \ref{THM1}, we introduce new functions. For $\varphi\in\mathcal{R}$, $\alpha\in\mathbb{R}_{\geq0}$ and $|t|\gg1$ we set 
\begin{align*}
\xi(\rho)&:=\begin{cases} 1, & \rho\in[0,1], \\ 0, & \rho\in[1,\infty).  \end{cases} \\
\varphi_\alpha(\rho)&:=\begin{cases} 1, & \rho\in[0,1-(2|t|^\alpha)^{-1}], \\ \varphi(1+(\rho-1)|t|^\alpha), & \rho\in[1-(2|t|^\alpha)^{-1},1+|t|^{-\alpha}], \\ 0, & \rho\in[1+|t|^{-\alpha},\infty), \end{cases}\\ 
\varphi_{0\alpha}(\rho)&:=1-\varphi_\alpha(1/\rho).
\end{align*}
Then these function have the following properties: 
\begin{lem}[{\cite[(12)--(15)]{GD1}}]\label{CFF}
For any $\alpha\in\mathbb{R}_{\geq0}$ and $\varphi\in\mathcal{R}$ we satisfy the following statements (i)--(iv):  
\begin{enumerate}[(i)]
\item $\varphi_\alpha,\varphi_{0\alpha}\in\mathcal{R}$.  
\item $(\varphi_\alpha-\xi)(\rho)=0, \ (\varphi_{0\alpha}-\xi)(\rho)=0$ for $\rho\in[0,1-(2|t|^\alpha)^{-1}]\cup[1+|t|^{-\alpha},\infty)$. 
\item  $ \varphi_\alpha^{(j)}(\rho)=0, \:\: \varphi_{0\alpha}^{(j)}(\rho)=0$ for $j\in\mathbb{Z}_{\geq1}$ and $\rho\in[0,1-(2|t|^\alpha)^{-1}]\cup[1+|t|^{-\alpha},\infty)$.
\item $\varphi_\alpha^{(j)}(\rho)\ll|t|^{\alpha j}, \ \varphi_{0\alpha}^{(j)}(\rho)\ll|t|^{\alpha j}, \ \|\varphi_\alpha^{(j)}\|_1\ll|t|^{\alpha(j-1)}, \ \|\varphi_{0\alpha}^{(j)}\|_1\ll|t|^{\alpha(j-1)}$ for $\rho\in[0,\infty)$ and $j\in\mathbb{Z}_{\geq0}$. 
\end{enumerate}
\end{lem}

\section{Proof of Theorem \ref{THM1}}\label{THM1P}
First for $s=\sigma+it: \sigma\in[0,1], |t|\gg1$, we shall use Cauchy's integral theorem in the region $$D_\sigma=\{w\in\mathbb{C}\mid w=u+iv,\;-1/2-\sigma\leq u\leq 3/2-\sigma,\;v\in\mathbb{R}\}.$$
Then from \eqref{FDFE} and Lemma \ref{LIB} the following lemma is obtained: 
\begin{pro}\label{PRO1}
For any $m\in\mathbb{Z}_{\geq0}$, $F\in\mathcal{S}$, $s=\sigma+it: \sigma\in[0,1], |t|\gg1$, $\varphi\in\mathcal{R}$ and $x\in\mathbb{R}_{>0}$ we have
\begin{align*}
F^{(m)}(s)=G_m(s;x,\varphi)+\chi_F(s)\sum_{r=0}^m(-1)^r\binom{m}{r}H_r(1-s;1/x,\varphi_0)
\end{align*}
where $G_r(s;x,\varphi)$ and $H_r(s;x,\varphi)$ are given by
\begin{align}
&\hspace{-1em}{G_r}(s;x,\varphi)\notag\\
=&\frac{1}{2\pi i}\int_{(\frac{3}{2}-\sigma)}\frac{{g_{F}}(s+w)}{{g_{F}}(s)}\frac{K_{\varphi}(w)}{w}{\frac{\chi_F^{(m-r)}}{\chi_F}}(1-(s+w)){F^{(r)}}(s+w)\times\notag\\
&\times\prod_{j=1}^q\frac{\Gamma(\lambda_j(s+w)+{\mu_j})}{\Gamma(\lambda_js+{\mu_j})}
(Q^{\frac{2}{d_F}}xe^{-i\frac{\pi}{2}{\rm sgn}(t)})^{\frac{d_F}{2}w}dw, \label{GSDEF}\\
&\hspace{-1em}H_r(s;x,\varphi)\notag\\
=&\frac{1}{2\pi i}\int_{(\frac{3}{2}-\sigma)}\frac{\overline{g_{F}}(s+w)}{\overline{g_{F}}(s)}\frac{K_{\varphi}(w)}{w}{\frac{\chi_F^{(m-r)}}{\chi_F}}(1-(s+w))\overline{F^{(r)}}(s+w)\times\notag\\
&\times\prod_{j=1}^q\frac{\Gamma(\lambda_j(s+w)+\overline{\mu_j})}{\Gamma(\lambda_js+\overline{\mu_j})}(Q^{\frac{2}{d_F}}xe^{-i\frac{\pi}{2}{\rm sgn}(t)})^{\frac{d_F}{2}w}dw \label{HSDEF}
\end{align}
respectively. Here $(3/2-\sigma)$ denotes $\left\{3/2-\sigma+iv\mid v\in\mathbb{R}\right\}$. 
\end{pro}
\begin{proof}
For $r\in\{0,1,\dots,m\}$ and $|v|\gg|t|$ let
\begin{align}
I_r(v)=\:&\frac{1}{2\pi i}\int_{-1/2-\sigma+iv}^{3/2-\sigma+iv}\frac{g_{F}(s+w)}{g_{F}(s)}\frac{K_\varphi(w)}{w}\frac{\chi_F^{(m-r)}}{\chi_F}(s+w)F^{(r)}(s+w)\times\notag\\
&\times\prod_{j=1}^q\frac{\Gamma(\lambda_j(s+w)+\mu_j)}{\Gamma(\lambda_js+\mu_j)}
(Q^{\frac{2}{d_F}}xe^{-i\frac{\pi}{2}{\rm sgn}(t)})^{\frac{d_F}{2}w}dw. \label{DFIV}
\end{align}

First we shall show that the integrand of \eqref{DFIV} is holomorphic in $D_\sigma\setminus\{0\}$. From Lemma \ref{KPL}, $K_\varphi(w)(Q^{{2}/{d_F}}xe^{-i\pi{\rm sgn}(t)/2})^{{d_Fw}/{2}}$ is holomorphic in $D_\sigma$. Since $F^{(m)}(w+s)$ has pole of order $p_F+m$ at most at $w=1-s$, we see that 
\begin{align}
(s+w-1)^{p_{F}+m}F^{(m)}(s+w) \label{GARR}
\end{align}
is holomorphic in $w\in D_\sigma$. 

Here we shall consider the holomorphicity of gamma-factor of \eqref{DFIV}. In the case of ${\rm Re\:}\mu_j>\lambda_j/2$, since ${\rm Re}(\lambda_j(s+w)+\mu_j)\geq -\lambda_j/2+{\rm Re\:}\mu_j>0$ for $w\in D_\sigma$ we see that $A_j(s+w)\Gamma(\lambda_j(s+w)+\mu_j)=\Gamma(\lambda_j(s+w)+\mu_j)$ is holomorphic in $D_\sigma$. On the other hand, in the case of ${\rm Re\;}\mu_j\leq\lambda_j/2$, we have ${\rm Re}(\lambda_j(s+w)+\mu_j+[\lambda_j/2-{\rm Re\;}\mu_j]+1)=1-\{\lambda_j/2-{\rm Re\;}\mu_j\}>0$ for $w\in D_\sigma$. Hence the functional equation for $\Gamma(s)$ implies that
\begin{align}
A_j(s)\Gamma(\lambda_j(s+w)+\mu_j)=\Gamma(\lambda_j(s+w)+\mu_j+[\lambda_j/2-{\rm Re\:}\mu_j]+1) \label{GBRR}
\end{align} 
is holomorphic in $D_\sigma$. 

Next we consider the existence of pole for $(\chi_F^{(m-r)}/\chi_F)(s+w)$ in $D_{\sigma}$. In the case of ${\rm Re\;}\mu_j> \lambda_j/2$, since 
\begin{align*}
&{\rm Re}(-s-(\mu_j+n)/\lambda_j)\leq-\sigma-({\rm Re\;}\mu_j)/\lambda_j<-1/2-\sigma,\\
&{\rm Re}(1-s+(\overline{\mu_j}+n)/\lambda_j)\geq1-\sigma+{\rm Re\;}\mu_j/\lambda_j> 3/2-\sigma
\end{align*}
for $n\in\mathbb{Z}_{\geq0}$, we see that $(\chi_F^{(m-r)}/\chi_F)(s+w)$ does not have pole in $D_{\sigma}$. On the other hand, in the case of ${\rm Re\;}\mu_j\leq\lambda_j/2$, since  
\begin{align*}
&{\rm Re}(-s-(\mu_j+n)/\lambda_j)\in\begin{cases}[-\sigma-1/2,-\sigma], & n\in\mathbb{Z}_{\geq0}\cap\mathbb{Z}_{\leq[\lambda_j/2-{\rm Re\;}\mu_j]}, \\ (-\infty,-1/2-\sigma), & n\in\mathbb{Z}_{>[\lambda_j/2-{\rm Re\;}\mu_j]}, \end{cases} \\ %and   
&{\rm Re}(1-s+(\overline{\mu_j}+n)/\lambda_j)\in\begin{cases} [1-\sigma,3/2-\sigma], & n\in\mathbb{Z}_{\geq0}\cap\mathbb{Z}_{\leq[\lambda_j/2-{\rm Re\;}\mu_j]},   \\  (3/2-\sigma), & n\in\mathbb{Z}_{>[\lambda_j/2-{\rm Re\;}\mu_j]}, \end{cases} 
\end{align*}
from Lemma \ref{CHIDF} the points $w=-s-(\mu_j+n)/\lambda_j,\;1-s+(\overline{\mu_j}+n)/\lambda_j$ for $n\in\mathbb{Z}_{\geq0}\cap\mathbb{Z}_{\leq[\lambda_j/2-{\rm Re\;}\mu_j]}$
are pole of order $(m-r)$ for $(\chi_F^{(m-r)}/\chi_F)(s+w)$ in $D_\sigma$. Therefore  
\begin{align}
\prod_{j=1}^q(A_j(s+w)\overline{A_j}(1-(s+w)))^{m-r}\frac{\chi_F^{(m-r)}}{\chi_F}(s+w) \label{GCRR}
\end{align}
is holomorphic in $D_{\sigma}$. Combining \eqref{GARR}--\eqref{GCRR} we find that the integrand of \eqref{HSDEF} is holomorphic $w\in D_\sigma\setminus\{0\}$. 

Next we shall show $I_r(v)\to0$ when $v\to\infty$. Trivial estimate gives  
\begin{align}
\frac{g_{F}(s+w)}{g_{F}(s)}\ll|v|^{2(p_{F}+m)+(m+1)f_F} \label{GGEE}
\end{align}
where $f_F=2\sum_{j=1}^q\max\{0,[\lambda_j/2-{\rm Re\:}\mu_j]\}$. 
To obtain an estimate of $F^{(r)}(s+w)$ in $w\in D_\sigma$, we shall consider estimates of $\chi_F(s)$ and $(\chi_F^{(r)}/\chi_F)(s)$. By the same method of \eqref{G1A3} and $it=|t|e^{i\frac{\pi}{2}{\rm sgn}(t)}$, we have
\begin{align}
\frac{\Gamma(\lambda_j(1-s)+\overline{\mu_j})}{\Gamma(\lambda_js+\mu_j)}
&=\frac{(-i\lambda_jt)^{\lambda_j(1-s)+\overline{\mu_j}-1/2}e^{i\lambda_jt}}{(i\lambda_j t)^{\lambda_js+\mu_j-1/2}e^{-i\lambda_jt}}(1+O(|t|^{-1}))\notag\\
&=(\lambda_j|t|)^{\lambda_j(1-2s)-2i{\rm Im\:}\mu_j}e^{i(2\lambda_j+\frac{\pi}{2}(-\lambda_j-2{\rm Re\:}\mu_j+1){\rm sgn}(t))}(1+O(|t|^{-1}))\notag\\
&=(\lambda_j|t|)^{\lambda_j(1-2\sigma)}e^{i\theta_j(t)}(1+O(|t|^{-1}))\label{GFAF0}
\end{align}
where $\theta_j(t)=2\lambda_j+{\rm sgn}(t)\cdot(-\lambda_j-2{\rm Re\:}\mu_j+1)\pi/2-\log(\lambda_j|t|)^{2(\lambda_j+{\rm Im\:}\mu_j)}$. Hence we obtain 
\begin{align}
\chi_F(s)=&\omega {C_F}^{\frac{1}{2}-\sigma}e^{i\theta_F(t)}|t|^{d_F(\frac{1}{2}-\sigma)}(1+O(|t|^{-1}))\label{CHI0E}
\end{align}
where $\theta_F(t)=d_F+{\rm sgn}(t)\cdot(-d_F/2-e_F+q)\pi/2-\log(C_F'|t|^{d_F+e_F})$ and $C_F'=(Q\prod_{j=1}^q\lambda_j^{\lambda_j+{\rm Im\:}\mu_j})^2$. Combining \eqref{FDFE}, \eqref{CHI0E} and Lemma \ref{CHIDF} we have 
\begin{align*}
F^{(r)}(s)\ll |t|^{d_F(\frac{1}{2}-\sigma)}\sum_{j=0}^r(\log|t|)^{r-j}|F^{(j)}(1-\sigma+it)|\ll |t|^{d_F(\frac{1}{2}-\sigma)}(\log|t|)^r 
\end{align*}
for $s\in\mathbb{C}: \sigma\in\mathbb{R}_{<0}, |t|\gg1$. Phragm\'{e}n-Lindel\"{o}f theorem implies 
\begin{align}
F^{(r)}(s+w)\ll|v|^{\frac{d_F}{2}(\frac{3}{2}-(\sigma+u))}(\log|v|)^{r} \label{DFEE}
\end{align}
uniformly for $u\in[-1/2-\sigma,3/2-\sigma]$. By \eqref{GGEE}, \eqref{DFEE}, Lemmas \ref{CHIDF}, \ref{GFEST} we get
\begin{align*}
I_r(v)
&\ll\int_{-1/2-\sigma}^{3/2-\sigma}|v|^{2(p_F+m)+(m+1)f_F}\frac{\|\varphi^{(l+1)}\|_1}{|v|^{l+1}}(\log|v|)^r\times\\
&\quad\;\times |v|^{\frac{d_F}{2}(\frac{3}{2}-(\sigma+u))}(\log|v|)^{m}|v|^{\frac{d_F(\sigma+u)+e_F-q}{2}}du\\
&\ll|v|^{\frac{3d_F}{4}+\frac{e_F-q}{2}+2(p_F+m)+(m+1)f_F-l-1}(\log|v|)^m\|\varphi^{(l+1)}\|_1
\end{align*}
when $|v|\gg|t|$. Choosing $l\in\mathbb{Z}_{>M_F}$ we find that $I_r(v)\to0$ when $v\to\infty$, where $M_F=3d_F/4+(e_F-q)/2+2(p_F+m)+(m+1)f_F$. Cauchy's integral theorem gives
\begin{align}
F^{(m)}(s)=&\frac{1}{2\pi i}\left(\int_{(3/2-\sigma)}-\int_{(-1/2-\sigma)}\right)\frac{g_{F}(s+w)}{g_{F}(s)}\frac{K_\varphi(w)}{w}F^{(m)}(s+w)\times\notag\\
&\times\prod_{j=1}^q\frac{\Gamma(\lambda_j(s+w)+\mu_j)}{\Gamma(\lambda_js+\mu_j)}(Q^{\frac{2}{d_F}}xe^{-i\frac{\pi}{2}{\rm sgn}(t)})^{\frac{d_F}{2}w}dw. \label{FMRES}
\end{align}
Here the first term of right-hand side of \eqref{FMRES} is
\begin{align}
=G_m(s;x,\varphi). \label{FMRES1}
\end{align}
Since \eqref{FCHI} implies  
\begin{align*}
\frac{\chi_F(s+w)}{\chi_F(s)}=\frac{1}{Q^{2w}}\prod_{j=1}^q\frac{\Gamma(\lambda_j(1-s-w)+\overline{\mu_j})}{\Gamma(\lambda_j(1-s)+\overline{\mu_j})}\times \prod_{j=1}^q\frac{\Gamma(\lambda_js+\mu_j)}{\Gamma(\lambda_j(s+w)+\mu_j)},
\end{align*}
combining this formula and \eqref{FDFE} we obtain
\begin{align*}
&\hspace{-1em}F^{(m)}(s+w)\prod_{j=1}^q\frac{\Gamma(\lambda_j(s+w)+\mu_j)}{\Gamma(\lambda_js+\mu_j)}\\
=&\frac{\chi_F(s)}{Q^{2w}}\prod_{j=1}^q\frac{\Gamma(\lambda_1(1-s-w)+\overline{\mu_1})}{\Gamma(\lambda_1(1-s)+\overline{\mu_1})}\sum_{r=0}^m(-1)^{r}\binom{m}{r}\frac{\chi_F^{(m-r)}}{\chi_F}(s+w)\overline{F^{(r)}}(1-s-w).
\end{align*}
By replacing $w\mapsto-w$ and  using this formula, \eqref{KPW1} and $\overline{g_F}(1-s)=g_F(s)$, the second term of right-hand side of \eqref{FMRES} is 
\begin{align}
=&\chi_F(s)\sum_{r=0}^m(-1)^{r}\binom{m}{r}\times\frac{1}{2\pi i}\int_{(3/2-(1-\sigma))}\frac{\overline{g_{F}}(1-s+w)}{\overline{g_{F}}(1-s)}\frac{K_{\varphi_0}(w)}{w}\times\notag\\
&\times\overline{F^{(r)}}(1-s+w)\frac{\chi_F^{(m-r)}}{\chi_F}(1-(1-s+w))\prod_{j=1}^q\frac{\Gamma(\lambda_j(1-s+w)+\overline{\mu_j})}{\Gamma(\lambda_j(1-s)+\overline{\mu_j})}\times\notag\\
&\times(Q^{\frac{2}{d_F}}x^{-1}e^{-i\frac{\pi}{2}{\rm sgn}(-t)})^{\frac{d_F}{2}w}dw\notag\\
=&\chi_F(s)\sum_{r=0}^m(-1)^{r}\binom{m}{r}H_r(1-s;1/x,\varphi_0). \label{FMICL}
\end{align}
Therefore, from \eqref{FMRES} and  \eqref{FMICL} Proposition \ref{PRO1} is obtained.  
\end{proof}

Next applying residue theorem to the functions ${G}_r(s;x,\varphi)$ and ${H}_r(s;x,\varphi)$, then these functions are approximated as follows:
\begin{pro}\label{PRO2}
For any $F\in\mathcal{S}$, $m\in\mathbb{Z}_{\geq0}$, $r\in\left\{0,\dots,m\right\}$, $s=\sigma+it: \sigma\in[0,1], |t|\gg1$, $\varphi\in\mathcal{R}$, $l\in\mathbb{Z}_{>M_F}$ and $x, y\in\mathbb{R}_{>0}: \sqrt{C_F}(x|t|)^{\frac{d_F}{2}}=y$, we have   
\begin{align}
&\hspace{-1em}{G_r}(s;x,\varphi)\notag\\
=&\sum_{n\leq 2y}\frac{{a_F(n)}(-\log n)^r}{n^s}\sum_{j=0}^l\varphi^{(j)}\left(\frac{n}{y}\right)\left(-\frac{n}{y}\right)^j{\gamma_{j}^{(m-r)}}\left(s;\frac{1}{(\lambda_1^{\lambda_1}\cdots{\lambda_r}^{\lambda_r})^{\frac{2}{d_F}}|t|}\right)+\notag\\
&+O(y^{1-\sigma}(\log y)^{r+\max\{p_F-1,0\}}|t|^{-\frac{l}{2}}(\log|t|)^{m-r}\|\varphi^{(l+1)}\|) \label{AGRB} \\
&\hspace{-1em}{H_r}(s;x,\varphi)\notag\\
=&\sum_{n\leq 2y}\frac{\overline{a_F(n)}(-\log n)^r}{n^s}\sum_{j=0}^l\varphi^{(j)}\left(\frac{n}{y}\right)\left(-\frac{n}{y}\right)^j{\delta_{j}^{(m-r)}}\left(s;\frac{1}{(\lambda_1^{\lambda_1}\cdots{\lambda_r}^{\lambda_r})^{\frac{2}{d_F}}|t|}\right)+\notag\\
&+O(y^{1-\sigma}(\log y)^{r+\max\{p_F-1,0\}}|t|^{-\frac{l}{2}}(\log|t|)^{m-r}\|\varphi^{(l+1)}\|), \label{AHRB}
\end{align}
where $M_F$ is some positive constant, and ${{\gamma}_{j,r}}(s;\rho)$, ${{\delta}_{j,r}}(s;\rho)$ are given by \eqref{DEFGM}, \eqref{DEFDM} respectively.
\end{pro}

\begin{proof}
In order to show \eqref{AGRB}, we use \eqref{KPW2} and write ${F^{(r)}}(s)=(\sum_{n\leq\rho y}+\sum_{n>\rho y})\times\times{a_F(n)}(-\log n)^rn^{-s}$ for ${\rm Re\;}s>1$ and $\rho\in\mathbb{R}_{>0}$. Then 
\begin{align}
{G_r}(s;x,\varphi)=I_1+I_2 \label{GHDIV}
\end{align}
where $I_1, I_2$ are given by
\begin{align}
I_1=&\frac{1}{2\pi i}\int_{(\frac{3}{2}-\sigma)}\frac{{g_{F}}(s+w)}{{g_{F}}(s)}
\left(\int_0^\infty\varphi^{(l+1)}(\rho)\rho^{w+l}\sum_{n\leq\rho y}\frac{{a_F(n)}(-\log n)^r}{n^{s+w}}d\rho\right)
\times\notag\\
&\times\frac{(-1)^{l+1}}{w\cdots(w+l)}{\frac{\chi_F^{(m-r)}}{\chi_F}}(s+w)\prod_{j=1}^q\frac{\Gamma(\lambda_j(s+w)+{\mu_j})}{\Gamma(\lambda_1s+{\mu_1})}(Q^{\frac{2}{d_F}}xe^{-i\frac{\pi}{2}{\rm sgn}(t)})^{\frac{d_F}{2}w}dw, \label{J1DEF}\\
I_2=&\frac{1}{2\pi i}\int_{(\frac{3}{2}-\sigma)}\frac{{g_{F}}(s+w)}{{g_{F}}(s)}\left(\int_0^\infty\varphi^{(l+1)}(\rho)\rho^{w+l}\sum_{n>\rho y}\frac{{a_F(n)}(-\log n)^r}{n^{s+w}}d\rho\right)\times\notag\\
&\times\frac{(-1)^{l+1}}{w\cdots(w+l)}{\frac{\chi_F^{(m-r)}}{\chi_F}}(s+w)\prod_{j=1}^q\frac{\Gamma(\lambda_j(s+w)+{\mu_j})}{\Gamma(\lambda_js+{\mu_j})}
(Q^{\frac{2}{d_F}}xe^{-i\frac{\pi}{2}{\rm sgn}(t)})^{\frac{d_F}{2}w}dw. \label{J2DEF}
\end{align}
To approximate $I_1$ and $I_2$, we define $L_{\pm j}$, $C_j$ ($j=1,2$) as 
\begin{align*}
L_{\pm j}=\{\sigma_j\pm iv\mid v\in[\sqrt{t},\infty)\}, \quad
C_j=\{\sigma_j+\sqrt{|t|}e^{-i\pi(\pm1/2+\theta)}\mid \theta\in[0,1]\},
\end{align*}
respectively, where $\sigma_1=-1/2-\sigma$ and $\sigma_2=3/2-\sigma$. The residue theorem gives 
\begin{align}
I_1=I_1'+{\rm Res\;}(I_1,\mathcal{F}), \quad I_2=I_2' \label{IJDEF}
\end{align}
where $I_1', I_2'$ are \eqref{J1DEF} replaced by $L_{-1}+C_1+L_{+1}$, $L_{-2}+C_2+L_{+2}$ respectively, and ${\rm Res\;}(I_1,\mathcal{F})$ is the sum of residue for the integrand of \eqref{J1DEF} in $\mathcal{F}$. Here by using the result 
\begin{align*}
\int_\mu^\infty\varphi^{(l+1)}(\rho)\rho^{w+l}d\rho=&\sum_{j=0}^l\varphi^{(l-j)}(\mu)\mu^{w+l-j}(-1)^{l-j}(w+l-j+1)\cdots(w+l)+\\
&+(-1)^{l+1}w\cdots(w+l)\int_\mu^\infty\varphi(\rho)\rho^{w-1}d\rho
\end{align*}
for $\mu\in\mathbb{R}_{\geq0}$ (see p.337 of \cite{GD1}) and the residue theorem, ${\rm Res\;}(I_1,\mathcal{F})$ is written as  
\begin{align}
&\hspace{-1em}{\rm Res\;}(I_1,\mathcal{F})\notag\\
=&\sum_{n\leq2y}\frac{{a_F(n)}(-\log n)^r}{n^s}\times\frac{1}{2\pi i}\int_{\mathcal{F}}\frac{{g_F}(s+w)}{{g_F}(s)}\frac{(-1)^l}{w\cdots(w+l)}{\frac{\chi_F^{(m-r)}}{\chi_F}}(s+w)\times\notag\\
&\times\left(\int_{\frac{n}{y}}^\infty\varphi^{(l+1)}(\rho)\rho^{w+l}d\rho\right)\prod_{j=1}^q\frac{\Gamma(\lambda_j(s+w)+{\mu_j})}{\Gamma(\lambda_js+{\mu_j})}\left(Q^{\frac{2}{d_F}}\frac{x}{n^{\frac{2}{d_F}}}e^{-i\frac{\pi}{2}{\rm sgn}(t)}\right)^{\frac{d_F}{2}w}dw\notag\\
=&\sum_{n\leq2y}\frac{{a_F(n)}(-\log n)^r}{n^s}\sum_{j=0}^l\varphi^{(j)}\left(\frac{n}{y}\right)\left(-\frac{n}{y}\right)^j{\gamma_{j}^{(m-r)}}\left(s;Q^{\frac{2}{d_F}}\frac{x}{y^{\frac{2}{d_F}}}\right). \label{REI1}
\end{align}
  
Next we shall estimate $I_1'$ and $I_2'$ as the error term of approximate functional equation. Partial summation and the assumption \eqref{LANA} give 
\begin{align}
&\int_{0}^{\infty}\varphi^{(l+1)}(\rho)\rho^{w+l}\sum_{n\leq\rho y}\frac{a_F(n)(-\log n)^r}{n^{s+w}}d\rho\notag\\
&\ll\int_{1/2}^{2}|\varphi^{(l+1)}(\rho)|\rho^{u+l}\left( \frac{(\log\rho{y})^{r+\max\{p_F-1,0\}}}{(\rho y)^{\sigma+u-1}}+\int_1^{\rho{y}}\frac{(\log u)^{r+\max\{p_F-1,0\}}}{u^{\sigma+u}}du \right)d\rho\notag\\
&\ll y^{1-(\sigma+u)}(\log y)^{r+\max\{p_F-1,0\}}\|\varphi^{(l+1)}\|_1\label{IJGEA}
\end{align}
for ${\rm Re}(s+w)\leq-1/2$ and 
\begin{align}
&\int_{0}^{\infty}\varphi^{(l+1)}(\rho)\rho^{w+l}\sum_{n>\rho y}\frac{a_F(n)(-\log n)^r}{n^{s+w}}d\rho\notag\\
&\ll \int_{1/2}^2|\varphi^{(l+1)}(\rho)|\rho^{u+l}\left(\int_{\rho{y}}^{\infty}\frac{(\log u)^{r+\max\{p_F-1,0\}}}{u^{\sigma+u}}du\right)d\rho\notag\\
&\ll y^{1-(\sigma+u)}(\log y)^{r+\max\{p_F-1,0\}}\|\varphi^{(l+1)}\|_1. \label{IJGEB}
\end{align}
for ${\rm Re}(s+w)\geq3/2$. From Lemma \ref{CHIDF} and \ref{GFEST} we get the following estimate: 
\begin{align}
&\frac{(-1)^l}{w\cdots(w+l)}{\frac{\chi_F^{(m-r)}}{\chi_F}}(1-(s+w))\prod_{j=1}^q\frac{\Gamma(\lambda_j(s+w)+{\mu_j})}{\Gamma(\lambda_js+{\mu_j})}(Q^{\frac{2}{d_F}}xe^{-i\frac{\pi}{2}{\rm sgn}(t)})^{\frac{d_F}{2}w} \notag\\
&\ll \begin{cases} \displaystyle \frac{(\log|t|)^{m-r}}{|t|^{\frac{l+1}{2}}}(\sqrt{C_F}(x|t|)^{\frac{d_F}{2}})^u, & w\in C_1\cup C_2\cup\mathcal{F}, \\ \displaystyle \frac{(\log|v|)^{m-r}}{|v|^{l+1}}\dfrac{(1+|t+v|)^{\frac{d_F(\sigma+u)+e_F-q}{2}}}{|t|^{\frac{d_F\sigma+e_F-q}{2}}}(\sqrt{C_F}x^{\frac{d_F}{2}})^u, & w\in L_{\pm1}\cup L_{\pm2}. \end{cases} \label{IJGE1}
\end{align}
Trivial estimate gives 
\begin{align}
\frac{{g_F}(s+w)}{{g_F}(s)}\ll\begin{cases} 1, & w\in C_1\cup C_2\cup \mathcal{F}, \\ \dfrac{(1+|t+v|)^{2(m+p_{F})+(m+1)f_F}}{|t|^{2(m+p_{F})+(m+1)f_F}}, & w\in L_{\pm1}\cup L_{\pm2}. \end{cases} \label{IJGE3}
\end{align}
Hence by combining \eqref{IJGEA}--\eqref{IJGE3}, $I_1'$ is estimated as 
\begin{align}
I_1'&\ll\int_{C_{1}}\frac{(\log|t|)^{m-r}}{|t|^{\frac{l+1}{2}}}(\sqrt{C_F}(x|t|)^{\frac{d_F}{2}})^uy^{1-(\sigma+u)}(\log y)^{r+\max\{p_F-1,0\}}\|\varphi^{(l+1)}\|_1|dw|+\notag\\
&\quad\;+\int_{L_{\pm1}}\frac{(\log|v|)^{m-r}}{|v|^{l+1}}\frac{(1+|t+v|)^{\frac{d_F(\sigma+u)+e_F-q}{2}+2(m+p_{F})+(m+1)f_F}}{|t|^{\frac{d_F(\sigma+u)+e_F-q}{2}+2(m+p_{F})+(m+1)f_F}}\times \notag\\
&\quad\;\times(\sqrt{C_F}(x|t|)^{\frac{d_F}{2}})^{u} y^{1-(\sigma+u)}(\log y)^{r+\max\{p_F-1,0\}}\|\varphi^{(l+1)}\|_1dv\notag\\
&\ll y^{1-\sigma}(\log y)^{r+\max\{p_F-1,0\}}|t|^{-\frac{l}{2}}(\log|t|)^{m-r}\|\varphi^{(l+1)}\|_1 \label{JDE2}
\end{align}
under the condition $\sqrt{C_F}(x|t|)^{\frac{d_F}{2}}=y$, where the following estimate was used:  
\begin{align*}
&\left(\int_{\pm\sqrt{|t|}}^{\pm\frac{|t|}{2}}+\int_{\pm\frac{|t|}{2}}^{\pm2|t|}+\int_{\pm2|t|}^{\pm\infty}\right) 
\frac{(1+|t+v|)^{M_F-d_F}}{|t|^{M_F-d_F}}\frac{(\log|v|)^{m-r}}{|v|^{l+1}}dv
=:J_1+J_2+J_3.
\end{align*}
Since $1\ll1+|t+v|\ll|t|$ when $v\in[-2|t|,-|t|/2]\cup[|t|/2,2|t|]$ and 
\begin{align*}
1+|t+v|\asymp\begin{cases} |t| & \text{when\;} v\in[-|t|/2,-\sqrt{|t|}]\cup[\sqrt{|t|},|t|/2], \\  |v| & \text{when\;} v\in(-\infty,-2|t|]\cup[2|t|,\infty), \end{cases}
\end{align*}
$J_j$ ($j=1,2,3$) were estimated as
\begin{align*}
J_1&\ll(\log|t|)^{m-r}\int_{\pm\sqrt{|t|}}^{\pm\frac{|t|}{2}}\frac{dv}{|v|^{l+1}}\ll|t|^{-\frac{l}{2}}(\log|t|)^{m-r},\\
J_2&\ll|t|^{\max\{0,d_F-M_F\}-(l+1)}(\log|t|)^{m-r}\int_{\pm\frac{|t|}{2}}^{\pm2|t|}\frac{dv}{(1+|t+v|)^{\max\{0,d_F-M_F\}}}\\
&\ll|t|^{-l+\max\{0,d_F-M_F\}}(\log|t|)^{m-r}\ll|t|^{-\frac{l}{2}}(\log|t|)^{m-r},\\
J_3&\ll\frac{1}{|t|^{M_F-d_F}}\int_{\pm2|t|}^{\pm\infty}
\frac{(\log|v|)^{m-r}}{|v|^{l+1-M_F+d_F}}dv\ll|t|^{-l}(\log|t|)^{m-r},
\end{align*}
where $l$ was chosen as $l\in\mathbb{Z}_{\geq 2\max\{0,d_F-M_F\}}$. By the same discussion of estimate of $I_1'$, the estimate of $I_2'$ is obtained:
\begin{align}
I_2'\ll& y^{1-\sigma}(\log y)^{r+A}|t|^{-\frac{l}{2}}(\log|t|)^{m-r}\|\varphi^{(l+1)}\|_1. \label{JDE1} 
\end{align}
Therefore combining \eqref{GHDIV}--\eqref{REI1}, \eqref{JDE2}--\eqref{JDE1} we obtain \eqref{AGRB}. Since we can obtain \eqref{AHRB} from the same discussion in the above, Proposition \ref{PRO2} is showed.  
\end{proof}

Finally, for any $x\in\mathbb{R}_{>0}$ we choose parameters $y_1, y_2\in\mathbb{R}_{>0}$ such that $\sqrt{C_F}(x|t|)^{\frac{d_F}{2}}\\=y_1$ and $\sqrt{C_F}(x^{-1}|t|)^{\frac{d_F}{2}}=y_2$, that is, $y_1y_2=C_F|t|^{d_F}$. By combining Propositions \ref{PRO1}, \ref{PRO2} and using \eqref{CHI0E}, $F^{(m)}(s)$ is approximated as
\begin{align*}
&\hspace{-1em}
F^{(m)}(s)\notag\\
=&\sum_{n\leq 2y_1}\frac{a_F(n)(-\log n)^m}{n^s}\sum_{j=0}^l\varphi^{(j)}\left(\frac{n}{y_1}\right)\left(-\frac{n}{y_1}\right)^j\gamma_{j}^{(0)}\left(s;\frac{1}{(\lambda_1^{\lambda_1}\cdots\lambda_q^{\lambda_q})^{\frac{2}{d_F}}|t|}\right)+\notag\\
&+\chi_F(s)\sum_{r=0}^m(-1)^r\binom{m}{r}\sum_{n\leq 2y_2}\frac{\overline{a_F(n)}(-\log n)^r}{n^{1-s}}\sum_{j=0}^r\varphi_0^{(j)}\left(\frac{n}{y_2}\right)\left(-\frac{n}{y_2}\right)^j\times\notag\\
&\times{\delta}_{j}^{(m-r)}\left(1-s;\frac{1}{(\lambda_1^{\lambda_1}\cdots\lambda_q^{\lambda_q})^{\frac{2}{d_F}}|t|}\right)+O(y_1^{1-\sigma}(\log y_1)^{m+A}|t|^{-\frac{l}{2}}\|\varphi^{(l+1)}\|_1)+\notag\\
&+O\left(y_2^\sigma|t|^{d_F(\frac{1}{2}-\sigma)-\frac{l}{2}}\|\varphi_0^{(l+1)}\|_1\sum_{r=0}^m(\log y_2)^{r+A}(\log|t|)^{m-r}\right). %\label{PCAFE}
\end{align*}
Using Lemma \ref{RCDE} and dividing sums of $j\in\mathbb{Z}_{\geq0}$ to term of $j=0$ and sum of $j\in\mathbb{Z}_{\geq1}$, we complete proof of Theorem \ref{THM1}. 

\section{Proof of Theorem \ref{THM2}}\label{THM2P}
In order to prove Theorem \ref{THM2} from the approximate functional equation containing characteristic functions, we use Lemma \ref{CFF}. For any functions $X:[0,\infty)\to\mathbb{R}$, we define $M_{X}(s)$ to 
\begin{align*}
M_{X}(s):=&\sum_{n=1}^\infty\frac{a_F(n)(-\log n)^m}{n^s}X\left(\frac{n}{y_1}\right)+\\ &+\sum_{r=0}^m(-1)^r\binom{m}{r}\chi_F^{(m-r)}(s)\sum_{n=1}^\infty\frac{\overline{a_F(n)}(-\log n)^r}{n^{1-s}}X_0\left(\frac{n}{y_2}\right). 
\end{align*}
Replacing $\varphi\mapsto\varphi_{\alpha} \;(\alpha\in\mathbb{R}_{\geq0})$ in Theorem \ref{THM1} we can write 
\begin{align}
F^{(m)}(s)=M_{\varphi_{\alpha}}(s)+R_{\varphi_\alpha}(s)=M_\xi(s)+O(M_{\varphi_\alpha-\xi}(s)+R_{\varphi_\alpha}(s)) \label{NAF1}
\end{align}
Here $M_\xi(s)$ and $M_{\varphi_\alpha-\xi}(s)+R_{\varphi_\alpha}(s)$ are written as 
\begin{align}
M_{\xi}(s)=\sum_{n\leq y}\frac{a_F(n)(-\log n)^m}{n^s}+\sum_{r=0}^m(-1)^r\binom{m}{r}\chi_F^{(m-r)}(s)\sum_{n\leq y}\frac{\overline{a_F(n)}(-\log n)^m}{n^{1-s}} \label{NAF2}
\end{align}
and 
\begin{align*}
M_{\varphi_\alpha-\xi}(s)+R_{\varphi_\alpha}(s)
=&\:E(s)+\sum_{n=1}^\infty\frac{a_F(n)(-\log n)^m}{n^s}S_{0}\left(\frac{n}{y_1}\right)+\\
&+\chi_F(s)\sum_{r=0}^m(-1)^r\binom{m}{r}\sum_{n=1}^\infty\frac{\overline{a_F(n)}(-\log n)^r}{n^{1-s}}T_{m-r}\left(\frac{n}{y_2}\right)
\end{align*}
respectively, where $E(s)$, $S_{0}(\rho)$ and $T_{r}(\rho)$ are given by 
\begin{align*}
E(s)&=O\left(y_1^{1-\sigma}(\log y_1)^{m+\max\{p_F-1,0\}}|t|^{-\frac{l}{2}}\|\varphi_\alpha^{(l+1)}\|_1\right)+\notag\\
&\quad+O\left(y_2^\sigma|t|^{d_F(\frac{1}{2}-\sigma)-\frac{l}{2}}\|\varphi_{0\alpha}^{(l+1)}\|_1{\textstyle\sum_{r=0}^m}(\log y_2)^{r+\max\{p_F-1,0\}}(\log|t|)^{m-r}\right).\\
S_0(\rho)&=(\varphi_\alpha-\xi)(\rho)+\sum_{j=1}^l\varphi_\alpha^{(j)}(\rho)(-\rho)^j\gamma_{j}^{(0)}\left(s;\frac{1}{(\lambda_1^{\lambda_1}\cdots\lambda_r^{\lambda_r})^{\frac{2}{d_F}}|t|}\right),\\
T_{m-r}(\rho)&=(\varphi_{0\alpha}-\xi)(\rho)\frac{\chi_F^{(m-r)}}{\chi_F}(s)+
\\&\quad+
\sum_{j=1}^l\varphi_{0\alpha}^{(j)}(\rho)(-\rho)^j\delta_{j}^{(m-r)}\left(1-s;\frac{1}{(\lambda_1^{\lambda_1}\cdots\lambda_r^{\lambda_r})^{\frac{2}{d_F}}|t|}\right),
\end{align*}
Since $S_{0}(\rho)=0$ and $T_{m-r}(\rho)=0$ for $\rho\in[0,(1+|t|^{-\alpha})^{-1}]\cup[1+|t|^{-\alpha},\infty)$ by Lemmas \ref{RCDE} and \ref{CFF}, we have 
\begin{align}
E(s)&\ll y_1^{1-\sigma}(\log y_1)^{m+\max\{p_F-1,0\}}|t|^{(\alpha-\frac{l}{2})l}+\notag \\
&\quad+y_2^\sigma|t|^{d_F(\frac{1}{2}-\sigma)+(\alpha-\frac{1}{2})l}\textstyle\sum_{r=0}^m(\log y_2)^{r+\max\{p_F-1,0\}}(\log|t|)^{m-r}\label{MUUE}
\end{align}
for any $\alpha\in\mathbb{R}_{\geq0}$, and 
\begin{align}
S_0(\rho)&\ll1+\sum_{j=1}^l|t|^{\alpha j}|t|^{-\frac{j}{2}}\ll 1,\label{MS0E}\\
T_{m-r}(\rho)&\ll(\log|t|)^{m-r}+\sum_{j=1}^l|t|^{\alpha j}|t|^{-\frac{j}{2}}(\log|t|)^{m-r}\ll(\log|t|)^{m-r}
\label{MTRE}
\end{align}
for $\rho\in[(1+|t|^{-\alpha})^{-1}, 1+|t|^{-\alpha}]$ under the condition $\alpha\in[0,1/2]$. Now we choose $\alpha=1/2-\varepsilon$ and $l\in\mathbb{Z}_{\geq1/(2\varepsilon)}$. By the condition (d) of Selberg class and the estimates  
\eqref{CHI0E}, \eqref{MUUE}--\eqref{MTRE} and $(1+|t|^{-\alpha})y_1-(1+|t|^{-\alpha})^{-1}y_1\leq 2|t|^{-\alpha}y_1$, $M_{\varphi_\alpha-\xi}(s)+R_{\varphi_\alpha}(s)$ is estimated as 
\begin{align}
&\hspace{-1em}M_{\varphi_\alpha-\xi}(s)+R_{\varphi_\alpha}(s)\notag\\
\ll&\;y_1^{1-\sigma}(\log y_1)^{m+\max\{p_F-1,0\}}|t|^{(\alpha-\frac{1}{2})l}+\notag \\
&+y_2^\sigma|t|^{d_F(\frac{1}{2}-\sigma)+(\alpha-\frac{1}{2})l}%\textstyle
\sum_{r=0}^m(\log y_2)^{r+\max\{p_F-1,0\}}(\log|t|)^{m-r}+\notag\\
&+\sum_{(1+|t|^{-\alpha})^{-1}y_1\leq n\leq(1+|t|^{-\alpha})y_1}\frac{|a_F(n)|(\log n)^m}{n^\sigma}+\notag\\
&+|t|^{d_F(\frac{1}{2}-\sigma)}\sum_{r=0}^m(\log|t|)^{m-r}\sum_{(1+|t|^{-\alpha})^{-1}y_2\leq n\leq(1+|t|^{-\alpha})y_2}\frac{|\overline{a_F(n)}|(\log n)^r}{n^{1-\sigma}}+\notag\\
\ll &\; y_1^{1-\sigma+\varepsilon}|t|^{-\frac{1}{2}}+y_2^{\sigma+\varepsilon}|t|^{d_F(\frac{1}{2}-\sigma)-\frac{1}{2}}.
\label{NAF3}
\end{align}
Hence combining \eqref{NAF1}, \eqref{NAF2} and \eqref{NAF3} we complete the proof of Theorem \ref{THM2}. 

\bibliographystyle{alpha}

\begin{thebibliography}{20}
\bibitem{A&M} M. Aoki and M. Minamide, \textit{A zero density estimate for the derivatives of the Riemann zeta function}, Journal for Algebra and Number Theory Academia \textbf{2} (2012), 361--365.
\bibitem{C&N} K. Chandrasekharan and R. Narasimhan, \textit{The approximate functional equation for a class of zeta functions}, Math. Ann. \textbf{152} (1963), 30--64.
\bibitem{DEL} P. Deligne, \textit{La conjecture de Weil. I}, Publ. Math. Inst. Hautes \'{E}tudes Sci. \textbf{43} (1974), 273--307.
\bibitem{GD1} A. Good, \textit{Approximative Funktionalgleichungen und Mittelwerts\"{a}tze f\"{u}r Dirichletreihen, die Spitzenformen assoziiert sind}, Comment. Math. Helv. \textbf{50} (1975), 327--361.
\bibitem{LAN} E. Landau,  \textit{\"{U}ber die Anzahl der Gitterpunkte in Gewissen Bereichen. II},  Nachr. Ge Wiss G\"{o}ttingen (1915), 209--243.
\bibitem{RAN} R. A. Rankin, \textit{Contributions to the theory of Ramanujan's function $\tau(n)$ and similar functions. II. The order of the Fourier coefficients of integral modular forms}, Proc. Cambridge Phil. Soc. \textbf{35} (1939), 357--373.
\bibitem{SEL2} A. Selberg, \textit{Bemerkungen \"{u}ber eine Dirichletsche Reihe, die mit der Theorie der Modulformen nahe verbunden ist}, Arch. Math. Naturvid. \textbf{43} (1940), 47--50.
\bibitem{SEL} A. Selberg, \textit{Old and new conjectures and results about a class of Dirichlet series}, In Proc. Amalfi Conf. Analytic Number Theory, E. Bombieri et al. eds., 367--385, Universit\`{a}di Salerno 1992; Collected Papers, Vol. II, 47--63, Berlin-Heidelberg-New York 1991.
%\bibitem{YY1} Y. Yashiro, \textit{Zero-density estimates for L-functions attached to cusp forms}, preprint; arXiv:1310.0765.
%\bibitem{SPE} A. Speiser, \textit{Geometrisches zur Riemannschen Zetafunktion}, Math. Ann. \textbf{110} (1935), 514--521.
\bibitem{YY3} Y. Yashiro, \textit{Approximate functional equation and mean value formula for the derivative of $L$-functions attached to cusp forms}, Funct. Approx. Comment. Math. (1) \textbf{53} (2015), 97--122.
\bibitem{YY4} Y. Yashiro, \textit{Zeros of the derivates of L-functions attached to cusp forms}, Bulletin Polish Acad. Sci. Math. \textbf{64} (2016), 147--164. %; arXiv:1402.3800.
\end{thebibliography}

\end{document}